\documentclass[12pt]{amsart}
\usepackage{amssymb,latexsym,eufrak,amsmath,amscd, graphicx}
\usepackage[colorlinks=true,pagebackref,hyperindex]{hyperref}

\usepackage{amsfonts}
\usepackage{amscd}

\renewcommand{\phi}{\varphi}
\renewcommand{\epsilon}{\varepsilon}
\renewcommand{\theta}{\vartheta}

\def\AAA{{\mathbf A}}

\def\cO{\mathcal{O}}

\def\fra{\mathfrak{a}}
\def\frb{\mathfrak{b}}

\def\frm{\mathfrak{m}}

 \DeclareMathOperator{\lct}{lct}
 
 \DeclareMathOperator{\ord}{ord}
 
\newcommand{\llbracket}{[\negthinspace[}
\newcommand{\rrbracket}{]\negthinspace]}

\newtheorem{lemma}{Lemma}[section]
\newtheorem{theorem}[lemma]{Theorem}
\newtheorem{corollary}[lemma]{Corollary}
\newtheorem{proposition}[lemma]{Proposition}

\theoremstyle{definition}

\newtheorem{remark}[lemma]{Remark}

\theoremstyle{remark}
\newtheorem*{remark*}{Remark}
\newtheorem*{note*}{Note}

\begin{document}

\title{Log canonical thresholds on smooth varieties: the Ascending Chain Condition}

\author[L. Ein]{Lawrence~Ein}
\address{Department of Mathematics, University of
Illinois at Chicago, 851 South Morgan Street (M/C 249),
Chicago, IL 60607-7045, USA} 
\email{{\tt ein@math.uic.edu}}

\author[M. Musta\c{t}\u{a}]{Mircea~Musta\c{t}\u{a}}
\address{Department of Mathematics, University of Michigan, 530 Church Street, 
Ann Arbor, MI 48109, USA}
\email{{\tt mmustata@umich.edu}}

\begin{abstract}
Building on results of Koll\'{a}r, we prove Shokurov's ACC Conjecture
for log canonical thresholds on smooth varieties, and more generally,
on varieties with quotient singularities.
\end{abstract}

\thanks{2000\,\emph{Mathematics Subject Classification}.
 Primary 14E15; Secondary 14B05.
\newline The first author
  was partially supported by NSF grant DMS-0700774, and the second author was partially supported by
  NSF grant DMS-0758454, and
  by a Packard Fellowship}
\keywords{Log canonical threshold, log resolution}

\maketitle

\markboth{L.~Ein and M.~Musta\c t\u a}{Log canonical thresholds on smooth varieties}

\section{Introduction}

Let $X$ be a smooth variety over an algebraically closed field $k$, of characteristic zero.
If $f\in\cO(X)$ is a nonzero regular function on $X$, and $p\in X$ is a point
such that $f(p)=0$, then the log canonical threshold $\lct_p(f)$ is an invariant of the singularity at $p$ of
the hypersurface defined by $f$. This invariant plays a fundamental role in birational geometry. 
For an overview of the many contexts in which this invariant appears, and for some applications, see
for example \cite{EM}. The following is our main result.

\begin{theorem}\label{main}
For every $n\geq 1$, the set 
$${\mathcal T}_n:=\{\lct_p(f)\mid f\in\cO(X),\,f(p)=0, \,X\,\text{smooth},\,\dim(X)=n\}$$
satisfies the Ascending Chain Condition, that is, it contains no infinite strictly increasing 
sequences.
\end{theorem}

The above statement has been conjectured by Shokurov in \cite{Sho}. In fact, Shokurov's conjecture
is in a more general setting, allowing the ambient variety to have mild singularities.
The interest in this conjecture comes from the fact that its general singular version 
is related to the Termination of Flips Conjecture (see \cite{Birkar} for a statement in this direction).
While the above theorem does not have such strong consequences, we believe that it offers
strong evidence for the general case of the conjecture. Moreover, it suggests that this general case might not be out of reach.

As we will see in Proposition~\ref{reduction_quotient} below, every log canonical threshold on a 
variety with quotient singularities can be also written as a log canonical threshold on a smooth variety
of the same dimension. Therefore the above theorem also implies Shokurov's Conjecture
for log canonical thresholds on varieties with quotient singularities.

The first unconditional results on limit points of log canonical thresholds in arbitrary dimension 
have been proved in \cite{dFM}, using ultrafilter constructions. The key point was to show that given
a sequence of polynomials $f_m\in k[x_1,\ldots,x_n]$ such that $\lim_{m\to\infty}\lct_0(f_m)=c$,
one can construct a formal power series $f\in K\llbracket x_1,\ldots,x_n\rrbracket$ (for a suitable
field extension $K$ of $k$) such that $\lct_0(f)=c$ (one can define log canonical thresholds also for formal power series, and it seems that for questions involving limit points of log canonical thresholds,
this is the right framework).

The results in \emph{loc. cit}. have been reproved by Koll\'{a}r in \cite{Kol1}, replacing the
ultrafilter construction by a purely algebro-geometric one. 
His approach turned out to be better suited for ruling out increasing sequences of
log canonical thresholds. In fact, using a deep recent result in the Minimal Model Program 
from \cite{BCHM}, he showed that if (with the notation in the previous paragraph) the formal power series $f$ is such that its log canonical threshold
is computed by a divisor with center at the origin, then $\lct_0(f_m)\geq c$ for all $m$ large enough.
In this note we show that elementary arguments on log canonical thresholds allow us to reduce to this 
special case, and therefore prove Theorem~\ref{main} above.

\subsection{Acknowledgment}
We are grateful to Shihoko Ishii and Angelo Vistoli for useful discussions and correspondence, 
and to J\'{a}nos Koll\'{a}r for his comments and suggestions on a preliminary version of this note.

\section{Getting isolated log canonical centers}

Let $k$ be an algebraically closed field of characteristic zero. Fix $n\geq 1$, and let
$R=k\llbracket x_1,\dots,x_n\rrbracket$. We put $X={\rm Spec}(R)$.
Suppose that $f\in R$ is a nonzero element 
in the maximal ideal $\frm=(x_1,\ldots,x_n)$. 
Our goal in this section is to show that after possibly replacing $f$ by a suitable $f^r$, we can find a 
polynomial
$g\in k[x_1,\ldots,x_n]$
such that $\lct_0(f)=\lct_0(fg)$ and such that the log canonical threshold of
$fg$ can be computed by a divisor $E$ over $X$ with center at the closed point. 
Our main reference for log canonical thresholds in the formal power series setting
is \cite{dFM} (see also \cite{Kol2}, \S 8 for the basic facts in the more familiar setting
of varieties over $k$).

Let us fix a log resolution $\pi\colon Y\to X$ for the pair $(R, f\cdot\frm)$. 
Note that by the main result in  \cite{Temkin} such resolutions exist, and this is what allowed in \cite{dFM} the
extension of the usual results on log canonical thresholds to our setting.
We write
$f\cdot\cO_Y=\cO(-\sum_ia_iE_i)$, $\frm\cdot\cO_Y=\cO_Y(-\sum_ib_iE_i)$, and 
$K_{Y/X}=\cO_Y(-\sum_ik_iE_i)$.
Note that the center of $E_i$ on $X$ is the closed point (that we denote by $0$) if and only if $b_i>0$. 

The log canonical threshold of $f$ is given by
$$\lct_0(f)=\min_i\frac{k_i+1}{a_i}.$$
We say that a divisor $E_j$ \emph{computes the log canonical threshold}
$\lct_0(f)$ if $\lct_0(f)=\frac{k_j+1}{a_j}$. 
Let $I$ denote the set of those $j$ for which  $E_j$ has center equal to
 the closed point.
If there is $j\in I$ computing $\lct_0(f)$, then there is nothing we need to do.
We now assume that this is not the case. The following is our key observation.

\begin{proposition}\label{key}
With the above notation, suppose that $c:=\lct_0(f)<\frac{k_i+1}{a_i}$ for every $i\in I$.
If 
\begin{equation}\label{eq_key}
q:=\min\left\{\frac{1}{b_i}\left(\frac{k_i+1}{c}-a_i\right)\mid i\in I\right\},
\end{equation}
then $\lct_0(f\cdot\frm^q)=\lct_0(f)$, and there is a divisor $E$ over $X$
computing $\lct_0(f\cdot\frm^q)$,
with center equal to the closed point.
\end{proposition}

\begin{proof}
Note that our assumption implies $q>0$.
By construction $\pi$ is also a log resolution of $(X,f\cdot\frm^q)$. Using the above notation for
$f\cdot\cO_Y$, $\frm\cdot\cO_Y$, and $K_{Y/X}$, we see that
$$\lct_0(f\cdot\frm^q)=\min_i\left\{\frac{k_i+1}{a_i+qb_i}\right\}.$$
Note that if $i\not\in I$, then $b_i=0$ and $\frac{k_i+1}{a_i+qb_i}=\frac{k_i+1}{a_i}\geq c$.

On the other hand, if $i\in I$ then by the definition of $q$ we have
\begin{equation}\label{eq_key2}
\frac{k_i+1}{a_i+qb_i}\geq c.
\end{equation}
This shows that $\lct_0(f\cdot\frm^q)\geq c$.
Moreover, if $i\in I$ is such that the minimum in (\ref{eq_key}) is achieved, then
we have equality in (\ref{eq_key2}). This shows that $\lct_0(f\cdot\frm^q)=\lct_0(f)$, and there is a divisor
$E_i$ with center equal to the closed point, that computes $\lct_0(f\cdot\frm^q)$.
\end{proof}

Our next goal is to replace the ideal $\frm^q$ by the suitable power of a product
of very general linear forms. Suppose that $\fra$ is an ideal in $k[x_1,\ldots,x_n]$
generated by $h_1,\ldots,h_d$, and such that
$\lct_0(\fra)=c$. It is well-known that if  $N\geq c$ and 
$g_1,\ldots, g_N$ are general linear combinations of $h_1,\ldots,h_d$ with coefficients in $k$, 
then $\lct_0(g)=c/N$, where $g=g_1\cdot\ldots\cdot g_N$ (see, for example, Prop. 9.2.26 
in \cite{positivity} for a related statement). 
The proof of this result, however, relies on Bertini's Theorem, so it does not simply carry over
to our setting. When dealing with formal power series, we will argue by taking truncations.
Let us start with a small variation on the above-mentioned result.

\begin{lemma}\label{lem1}
Let $\fra$ and $\frb$ be nonzero ideals in the polynomial ring $k[x_1,\ldots,x_n]$, and
$\alpha$, $\beta>0$ such that $\lct_0(\fra^{\alpha}\cdot\frb^{\beta})=c_0$. Let $N$
be a positive integer such that $N/\beta\geq c_0$. If $h_1,\ldots,h_d$ generate $\frb$,
and if $g_1,\ldots,g_N$ are general linear combinations of the $h_i$ with coefficients in $k$, then
$\lct_0(\fra^{\alpha}\cdot g^{\beta/N})=c_0$, where $g=g_1\cdot\ldots\cdot g_N$.
\end{lemma}

\begin{proof}
The argument is the same one as in \emph{loc. cit}., but we recall it for completeness. Let $\mu\colon W\to \AAA^n$ be a log resolution of 
$\fra\cdot\frb$. Let us write $\fra\cdot\cO_W=\cO(-\sum_iu_iE_i)$, $\frb\cdot
\cO_W=\cO(-\sum_iw_iE_i)$, and $K_{W/\AAA^n}=\sum_ik_iE_i$. Therefore
$$c_0=\min_{i\in J}\left\{\frac{k_i+1}{\alpha u_i+\beta w_i}\right\},$$
where $J$ is the set of those $i$ such that $0\in \mu(E_i)$.

Since each $g_j$ is a general combination of the generators of $\frb$, 
it follows by Bertini's Theorem that $g_j\cdot\cO_W=\cO(-F_j-\sum_iw_iE_i)$,
where $F_1,\ldots, F_N$ are divisors with no common components amongst them or with the $E_i$,
such that
$\sum_iE_i+\sum_jF_j$ has simple normal crossings. In particular, $\mu$ is a log resolution of
$\fra\cdot g$. Since $g\cdot\cO_W=\cO(-\sum_iNw_iE_i-\sum_jF_j)$, and
$$\min\left\{
\min_{i\in J}\frac{k_i+1}{\alpha u_i+\beta w_i},\,\frac{1}{\beta/N}\right\}=c_0$$
(we use the hypothesis that $N/\beta\geq c_0$), it follows that $\lct_0(\fra^{\alpha}\cdot g^{\beta/N})
=c_0$.
\end{proof}

In the following proposition we assume that the ground field $k$ is uncountable.
In this case, we follow the standard terminology by saying that a very general point on an
irreducible algebraic variety over $k$ is a point that lies outside a countable union of
proper subvarieties.

\begin{proposition}\label{prop2} 
With the notation in Proposition~\ref{key},
let $N>qc$ be an integer. If 
$g_1,\ldots,g_N$ are very general linear linear forms with coefficients in $k$,
and $g=g_1\cdot\ldots\cdot g_N$, then $\lct_0(f\cdot g^{q/N})=c$, and there is a divisor $E$ over $X$
computing $\lct_0(f\cdot g^{q/N})$,
with center equal to the closed point.
\end{proposition}

\begin{proof}
For every $\ell$, let $f_{\leq\ell}$ denote the truncation of $f$ of degree $\leq\ell$. 
Let us put $c_{\ell}=\lct_0(f_{\leq \ell}\cdot\frm^q)$. It follows from Proposition~\ref{key} and 
Prop. 2.5 in \cite{dFM} that $\lim_{\ell\to\infty}c_{\ell}=c$. In particular,
if $\ell\gg 0$, then $N>q c_{\ell}$.
 For such $\ell$ we may apply the lemma
to deduce that
$\lct_0(f_{\leq \ell}\cdot g^{q/N})=c_{\ell}$ (note the since we are allowed to take the $g_i$
as very general linear combinations of $x_1,\ldots,x_n$, we can simultaneously 
apply the lemma for all $\ell$ as above).
Another application of Prop. 2.5 in \cite{dFM} now gives $\lct_0(f\cdot g^{q/N})=c$.

Suppose that $E$ is the divisor given by Proposition~\ref{key}, so it computes
$\lct_0(f\cdot\frm^q)$, and its center is equal to the closed point. 
We denote by $\ord_E$ the valuation of the fraction field of $R$ corresponding to $E$.
Since $g\in\frm^N$, we deduce
$\ord_E(f\cdot g^{q/N})\geq\ord_E(f\cdot\frm^q)$. Therefore we have
\begin{equation}\label{eq3}
c=\frac{\ord_E(K_{Y/X})+1}{\ord_E(f\cdot \frm^q)}\geq 
\frac{\ord_E(K_{Y/X})+1}{\ord_E(f\cdot g^{q/N})}\geq c,
\end{equation}
where the second equality follows from the fact that $\lct_0(f\cdot g^{q/N})=c$.
Therefore both inequalities in (\ref{eq3}) are equalities, and we see that $E$ computes
$\lct_0(f\cdot g^{q/N})$.
\end{proof}

\section{The proof of the ACC Conjecture}

Before giving the proof of Theorem~\ref{main}, let us describe the key construction
from \cite{Kol1}. 
For every nonnegative integer $m$, let ${\rm Pol}_{\leq m}$ denote the affine space
$\AAA_k^{N_m}$ (with $N_m={{n+m}\choose {m}}$), such that for every field extension
$K$ of $k$, the $K$-rational points of ${\rm Pol}_{\leq m}$ parametrize the polynomials in
$K[x_1,\ldots,x_n]$ of degree $\leq m$. We have obvious maps
$\pi_m\colon {\rm Pol}_{\leq m}\to {\rm Pol}_{\leq (m-1)}$ that correspond to truncation
of polynomials. 
Given a field extension $K$ of $k$, a formal power series $f\in K\llbracket x_1,\ldots,x_n\rrbracket$
corresponds to a sequence of morphisms ${\rm Spec}\,K\to {\rm Pol}_{\leq m}$ over 
${\rm Spec}\,k$, compatible via the truncation maps. We denote by $t_m(f)$ the corresponding element of ${\rm Pol}_{\leq m}(K)$. 

Suppose now that $(f_q)_q$ is a sequence of formal power series 
in $k\llbracket x_1,\ldots,x_n\rrbracket$, all being nonzero, and of positive order.
We consider sequences of 
irreducible closed subsets $Z_m
\subseteq {\rm Pol}_{\leq m}$ with the following properties:
\begin{enumerate}
\item[i)] For every $m$, there are infinitely many
$q$ such that $t_m(f_q)\in Z_m$.
\item[ii)] $Z_m$ is the Zariski closure of those $t_m(f)\in Z_m$.
\item[iii)] Each truncation morphism $\pi_m$ induces a dominant morphism $Z_m\to Z_{m-1}$.
\end{enumerate}

Such sequences can be constructed by induction.
We start the induction by taking $Z_0={\rm Pol}_{\leq 0}={\rm Spec}(k)$. If $Z_m$ is constructed, 
then we take $Z_{m+1}$ to be a minimal irreducible closet subset of $\pi_{m+1}^{-1}(Z_m)$
with the property that it contains $t_{m+1}(f_q)$ for infinitely many $q$. 
Properties i)-iii) are clear: note that the minimality assumption in the definition of $Z_{m+1}$ implies that the induced morphism
$Z_{m+1}\to Z_{m}$ is dominant.

Suppose that $(Z_m)_m$ is a sequence satisfying i)-iii) above.
 Let $\eta_m$ denote the generic point of 
$Z_m$, so the truncation maps induce embeddings $k(\eta_m)\hookrightarrow k(\eta_{m+1})$.
If $K$ is an algebraically closed field containing $\bigcup_mk(\eta_m)$,
we get corresponding maps ${\rm Spec}(K)\to {\rm Pol}_{\leq m}$ that are compatible via the
truncation morphisms. This corresponds to a formal power series $f\in K\llbracket x_1,\ldots,x_n\rrbracket$
such that $t_m(f)$ gives $\eta_m$.
Of course, this construction is not unique. However, whenever $f$ is obtained 
 from the sequence $(f_q)_q$ via such  a sequence $(Z_m)_m$, we  say that $f$ is \emph{a generic limit} of $(f_q)_q$. A trivial example is when $f_m=h$ for every $m$, in which case 
 each $Z_m$ is a point, and for every field extension $K$ of $k$, a generic limit of this sequence
 is given by the image of $f$ in $K\llbracket x_1,\ldots,x_n\rrbracket$.

A key property of generic limits that follows from construction is that for every $m$, there are infinitely many $q$ such that
$\lct_0(t_m(f))=\lct_0(t_m(f_q))$. It is easy to deduce from this that if $\lim_{q\to\infty}\lct_0(f_q)
=c$, then $\lct_0(f)=c$ (see Thm.~29 in \cite{Kol1} for details).

We start with an easy lemma describing the behavior of generic limits under multiplication.

\begin{lemma}\label{easy1}
Let $(f_q)_q$ and $(h_q)_q$ be two sequences of 
formal power series in $k\llbracket x_1,\ldots,x_n\rrbracket$.
If $f\in K'\llbracket x_1,\ldots,x_n\rrbracket$, and $h\in
K''\llbracket x_1,\ldots,x_n\rrbracket$ are generic limits of $(f_q)_q$, and respectively $(h_q)_q$,
then there is a field extension $K$ of both $K'$ and $K''$, such that $fh\in
K\llbracket x_1,\ldots,x_n\rrbracket$ is a generic limit
of the sequence $(f_qh_q)_q$.
\end{lemma}

\begin{proof}
Suppose that $Z'_m, Z''_m\subseteq {\rm Pol}_{\leq m}$ are sequences  of subsets satisfying i)- iii)
above, with respect to $(f_q)_q$, and respectively $(h_q)_q$. If $\eta'_m$ and $\eta''_m$ are the generic points of respectively $Z'_m$ and $Z''_m$, then $K'$ contains all $k(\eta'_m)$, and similarly,
$K''$ contains all $k(\eta''_m)$. 
For every $m$ we have a morphism $\phi_m\colon {\rm Pol}_{\leq m}\times
{\rm Pol}_{\leq m}
\to {\rm Pol}_{\leq m}$ that corresponds to multiplication, followed
by truncation up to degree $m$. It is clear that we have $\pi_m\circ\phi_m=
\phi_{m-1}\circ(\pi_m,\pi_m)$ for every positive $m$. If we take $Z_m:=\overline{\phi_m(Z'_m
\times Z''_m)}$, then $(Z_m)_m$
is a sequence of irreducible closed sets that satisfies i)-iii) with respect to the sequence
$(f_qh_q)_q$. Moreover, if $\eta_m$ is the generic point of $Z_m$, then we have
embeddings compatible with the maps induced by truncation 
$$k(\eta_m)\hookrightarrow  Q(k(\eta'_m)\otimes_kk(\eta''_m)),$$
where for a domain $S$ we denote by $Q(S)$ its fraction field. 
Since $k$ is algebraically closed, $K'\otimes_kK''$ is a domain. 
If $K$ is an algebraically closed extension of $Q(K'\otimes_kK'')$, then $K$ contains
$\bigcup_mk(\eta_m)$, and it follows from the construction that $fh$ is the generic limit
corresponding to the sequence $(Z_m)_m$, and to the embedding $\bigcup_mk(\eta_m)
\hookrightarrow K$.
\end{proof}

The following is the key result of Koll\'{a}r that we will use. Its proof uses
the difficult finite generation result of \cite{BCHM}.

\begin{proposition}\label{ingred} ${\rm (Prop.\,40}$, \cite{Kol1}${\rm )}$
Let $K\supseteq k$ be a field extension, and suppose that $F\in K\llbracket x_1,\ldots,x_n\rrbracket$
is a formal power series. Suppose that there is a divisor computing $\lct_0(F)$ with center at the closed point. If $Z_m\subseteq {\rm Pol}_{\leq m}$ is the $k$-Zariski closure of $t_m(F)$, then there is 
a positive integer $m$, and an open subset $U_m\subseteq Z_m$ such that for every
power series $G\in K\llbracket x_1,\ldots,x_n\rrbracket$ with $t_m(G)\in U_m$, we have
$\lct_0(F)=\lct_0(G)$.
\end{proposition}

Note that if $F$ is a generic limit of a sequence $(f_q)_q$ of power series in $k\llbracket
x_1,\ldots,x_n\rrbracket$, then the sets $Z_m$ in the above statement are the same as the sets
that come up in the definition of $F$ as a generic limit.
We can now prove our main result.

\begin{proof}[Proof of Theorem~\ref{main}]
We may assume that $k$ is uncountable, and
it is enough to show that the set
$${\mathcal T}'_n=\{\lct_0(f)\mid f\in k\llbracket x_1,\ldots,x_n\rrbracket,\, f\neq 0, \,\ord(f)\geq 1\}$$
contains no strictly increasing infinite sequences.  
This follows since we can extend scalars to an uncountable algebraically closed field, and we can
complete at the given point $p\in X$ (in fact, 
it is shown in \cite{dFM} that ${\mathcal T}_n={\mathcal T}'_n$,
and that this set is independent of the algebraically closed field $k$; however, we do not need this fact).

Let us suppose that $(f_q)_q$ is a sequence of formal power series in $k\llbracket
x_1,\ldots,x_n\rrbracket$, nonzero and of positive order, such that 
if we put $c_q=\lct_0(f_q)$, then
the corresponding sequence
$(c_q)_q$ is strictly increasing. We put $c:=\lim_{q\to\infty}c_q$. 
Let $f\in K\llbracket x_1,\ldots,x_n\rrbracket$ be a generic limit of the sequence $(f_q)_q$,
with $K$ an algebraically closed extension of $k$. As mentioned above, we have
$\lct_0(f)=c$.

If $\lct_0(f)$ is computed by some divisor $E$ with center equal to the closed point, then we are done.
Indeed, we apply Proposition~\ref{ingred} above to $F=f$ to get $m$ and an open subset 
$U_m\subseteq Z_m$ with the property that whenever $G\in K\llbracket x_1,\ldots,x_n\rrbracket$
satisfies $t_m(G)\in U_m$, we have $\lct_0(f)=\lct_0(G)$. On the other hand, by assumption 
$Z_m$ is the Zariski closure of those $t_m(f_q)\in Z_m$.
Therefore we can find $q$ with $t_m(f_q)\in U_m$, so that $c_q=c$, a contradiction.

Suppose now that $\lct_0(f)$ is not computed by any divisor with center equal to the closed point. 
In this case Propositions~\ref{key} and \ref{prop2} imply that we can find a positive rational number
$q$, a positive integer $N$  and $g\in k\llbracket x_1,\ldots,x_n\rrbracket$ 
such that $\lct_0(f\cdot g^{q/N})=c$, and there is a divisor $E$ computing $\lct_0(f\cdot g^{q/N})$,
and having center equal to the closed point. Let us write $q/N=s/r$, for positive integers $r$ and $s$.
It is clear that $\lct_0(f^rg^s)=c/r$, and the same divisor $E$ computes this log canonical threshold. 

On the other hand, by Lemma~\ref{easy1} the power series $f^rg^s$
is a generic limit of the sequence $(f_q^rg^s)_q$. Since $\lct_0(f^rg^s)$ is computed by a divisor
with center equal to the closed point, using Proposition~\ref{ingred} for $F=f^rg^s$, we deduce that
for some  $q$ we have $\lct_0(f_q^rg^s)=c/r$. Note now that $\lct_0(f_q^rg^s)
\leq \lct_0(f_q^r)=c_q/r$, so we have a contradiction. This completes the proof of the theorem.
\end{proof}

\begin{remark}\label{ideals}
While we stated Theorem~\ref{main} only for log canonical thresholds of principal ideals,
it is straightforward to extend the statement to arbitrary ideals. More precisely, the theorem
implies that for every $n\geq 1$, the set of all log canonical thresholds $\lct_p(X,\fra)$
satisfies the Ascending Chain Condition, when $X$ 
varies over the smooth $n$-dimensional varieties, and $\fra$ over the nonzero ideals on $X$ containing
$p$ in their support. Indeed, we may assume that $X$ is affine, in which case if $f_i$ is a general linear combination
of the generators of $\fra$, for $1\leq i\leq n$, and if $f=f_1\cdot\ldots\cdot f_n$, then
$\frac{1}{n}\cdot\lct_p(X,\fra)=\lct_p(X,f)$ (note that $\lct_p(X,\fra)\leq n$, and apply for example
Lemma~\ref{lem1}). Therefore each such $\lct_p(X,\fra)$ lies in $n\cdot{\mathcal T}_n$, and this set satisfies the Ascending Chain Condition by Theorem~\ref{main}.
\end{remark}

The following proposition allows us to reduce log canonical thresholds on varieties with quotient singularities to log canonical thresholds on smooth varieties. We say that a variety $X$ has 
quotient singularities at $p\in X$ if there is a smooth variety $U$, a finite group $G$ acting on
$U$, and a point $q\in V=U/G$ such that the two completions 
$\widehat{\cO_{X,p}}$ and $\widehat{\cO_{V,q}}$ are isomorphic as $k$-algebras. 
One can assume that $U$ is an affine space and that the action of $G$
is linear. Furthermore, one can assume that $G$ acts with no fixed points in codimension one
(otherwise, we may replace $G$ by $G/H$ and $U$ by $U/H$, where $H$ is generated by all pseudoreflections in $G$, and by Chevalley's theorem \cite{Chevalley}, the quotient 
$U/H$ is again an affine space). 
Using Artin's approximation results (see Corollary~2.6 in \cite{Artin}), it follows that 
there is an \'{e}tale neighborhood of $p$ that is also an \'{e}tale neighborhood of $q$.
In other words, there is a variety $W$, a point $r\in W$, and \'{e}tale maps
$\phi\colon W\to X$ and $\psi\colon W\to V$, such that $p=\phi(r)$ and $q=\psi(r)$. 
After replacing $\phi$ by the compositon
$$W\times_VU\to W\overset{\phi}\to X,$$
we may assume that in fact we have an \'{e}tale map $U/G\to X$ containing $p$ in its image,
with $U$ smooth, 
and such that $G$ acts on $U$ without fixed points in codimension one. 
This reinterpretation of the definition of quotient singularities seems to be well-known
to experts, but we could not find an explicit reference in the literature.

We say that $X$ has quotient singularities if it has quotient singularities at every point.

\begin{proposition}\label{reduction_quotient}
Let $X$ be a variety with quotient singularities, and let $\fra$ be a proper nonzero ideal on $X$. 
For every $p$ in the zero-locus $V(\fra)$ of $\fra$, there is a smooth variety $U$, a nonzero ideal
$\frb$ on $U$, and a point $q$ in $V(\frb)$ such that $\lct_p(X,\fra)=\lct_q(U,\frb)$.
\end{proposition}

\begin{proof}
Let us choose an \'{e}tale map $\phi\colon U/G\to X$ with $p\in {\rm Im}(\phi)$, where 
$U$ is a smooth variety, and $G$ is a finite group acting on $U$ 
without fixed points in codimension one. 
Let $\widetilde{\phi}\colon U\to X$ denote the composition of $\phi$ with the quotient map.
Since $G$ acts without fixed points in codimension one, $\widetilde{\phi}$ is
\'{e}tale in codimension one, hence $K_U=\widetilde{\phi}^*(K_X)$. It follows from 
Proposition~5.20 in \cite{KM} that if $\frb=\fra\cdot\cO_U$, then the pair $(X,\fra^q)$ is log canonical if and only if the pair
$(U,\frb^q)$ is log canonical (actually the result in \emph{loc. cit.} only covers the case when $\fra$
is locally principal, but one can easily reduce to this case, by taking a suitable product of general
linear combinations of the local generators of $\fra$). We conclude that there is a point $q\in
V(\frb)$ such that $\lct_p(X,\fra)=\lct_q(U,\frb)$.
\end{proof}

\begin{remark}\label{usual_definition}
At least over the complex numbers, one usually says that $X$ has quotient singularities at $p$
if the germ of analytic space $(X,x)$ is isomorphic to $M/G$, where $M$ is a complex manifold,
and $G$ is a finite group acting on $M$. It is not hard to check that in this context this definition is equivalent with the one we gave above.
\end{remark}

Combining Proposition~\ref{reduction_quotient}
 and Theorem~\ref{main} (see also Remark~\ref{ideals}), we deduce
Shokurov's ACC Conjecture on varieties with quotient singularities. 

\begin{corollary}\label{quotient}
For every $n\geq 1$, the set 
$$\{\lct_p(X,\fra)\mid X\,\text{has quotient singularities},\,\dim(X)=n,(0)\neq\fra,\,p\in V(\fra)\}$$
satisfies the Ascending Chain Condition.
\end{corollary}

\providecommand{\bysame}{\leavevmode \hbox \o3em
{\hrulefill}\thinspace}


\begin{thebibliography}{dFM}

\bibitem[Art]{Artin}
M.~Artin,  Algebraic approximation of structures over complete local rings,
 Inst. Hautes \'{E}tudes Sci. Publ. Math. \textbf{36} (1969), 23--58.

\bibitem[BCHM]{BCHM}
C.~Birkar, P.~Cascini, C.~Hacon and J.~M$^{\rm c}$Kernan, 
Existence of minimal models for varieties of log general type,
preprint available at math/0610203.

\bibitem[Bir]{Birkar}
C.~Birkar, Ascending chain condition for log canonical thresholds
and termination of log flips, Duke Math. J. \textbf{136} (2007),
173--180.

\bibitem[Che]{Chevalley}
C.~Chevalley, Invariants of finite groups generated by reflections,
Amer. J. Math. \textbf{77} (1955), 778--782.

\bibitem[dFM]{dFM}
T.~de Fernex and M.~Musta\c{t}\u{a}, Limits of log canonical thresholds, 
preprint available at 	arXiv:0710.4978.

\bibitem[EM]{EM}
L.~Ein and M.~ Musta\c{t}\u{a}, Invariants of singularities of pairs, in \emph{International Congress of Mathematicians},  Vol. II, 583--602, Eur. Math. Soc., Z\"{u}rich, 2006.

\bibitem[Kol1]{Kol1}
J.~Koll\'{a}r, Which powers of holomorphic functions are integrable?, 
preprint available at	arXiv:0805.0756.


\bibitem[Kol2]{Kol2}
J.~Koll\'{a}r, Singularities of pairs, in \emph{Algebraic geometry,
Santa Cruz 1995}, 221--286,  Proc. Symp. Pure
Math. 62, Part 1, Amer. Math. Soc., Providence, RI, 1997.

\bibitem[KM]{KM}
J.~Koll\'{a}r and S.~Mori, \emph{Birational geometry of algebraic varieties},
Cambridge Tracts in Mathematics 134, Cambridge University Press, Cambridge, 1998.

\bibitem[Laz]{positivity} R.~Lazarsfeld,
\emph{Positivity in algebraic geometry} II,
Ergebnisse der Mathematik und ihrer Grenzgebiete 49, Springer-Verlag, Berlin, 2004.



\bibitem[Sho]{Sho}
V.~V.~Shokurov,
 Three-dimensional log perestroikas. With an appendix
in English by Yujiro Kawamata, Izv. Ross. Akad. Nauk Ser. Mat.
\textbf{56} (1992), 105--203, translation in {Russian Acad. Sci.
Izv. Math.} \textbf{40} (1993), 95--202.


\bibitem[Tem]{Temkin}
M.~Temkin,  Desingularization of quasi-excellent schemes in
characteristic zero, Adv. Math. \textbf{219} (2008), 488--522.



\end{thebibliography}
\end{document}